\documentclass[a4paper,12pt]{amsart}
\usepackage[utf8]{inputenc}

\usepackage{url}
\usepackage[margin=1in]{geometry}  
\usepackage{epsfig}
\usepackage{color}
\usepackage{graphicx}              
\usepackage{amsmath}
\usepackage{amscd}               
\usepackage{amsfonts}
\usepackage{amssymb}             
\usepackage{amsthm}                
\usepackage{epsfig}
\usepackage{mathrsfs}
\usepackage{hyperref}
\usepackage{enumerate}
\usepackage{esint}

\newtheorem{theorem}{Theorem}

\newtheorem*{theorem*}{Theorem}
\newtheorem{lemma}[theorem]{Lemma}
\newtheorem*{lemma*}{Lemma}
\newtheorem{proposition}[theorem]{Proposition}
\newtheorem*{proposition*}{Proposition}
\theoremstyle{definition}

\theoremstyle{remark}

\newcommand{\add}{\operatorname{add}}
\newcommand{\sub}{\operatorname{sub}}

\newcommand{\Disc}{\operatorname{Disc}}
\newcommand{\RE}{\operatorname{Re}}

\newcommand{\vol}{\operatorname{vol}}

\newcommand{\Res}{\operatorname{Res}}

\newcommand{\E}{\mathbf{E}}     
\newcommand{\one}{\mathbf{1}}

\newcommand{\Stab}{\operatorname{Stab}}

\newcommand{\bR}{\mathbb{R}}

\newcommand{\zed}{\mathbb{Z}}

\newcommand{\GL}{\mathrm{GL}}

\newcommand{\SL}{\mathrm{SL}}
\newcommand{\SO}{\mathrm{SO}}

\newcommand{\sC}{{\mathscr{C}}}

\newcommand{\sF}{{\mathscr{F}}}

\newcommand{\sL}{{\mathscr{L}}}

\newcommand{\sS}{{\mathscr{S}}}

\title{Subconvexity of Shintani's zeta function}
\author{Robert D. Hough}
\address{Department of Mathematics, Stony Brook University, 100 Nicolls Road, 
Stony Brook, NY 11794}
\email{robert.hough@stonybrook.edu}
\author{Eun Hye Lee}
\address{Department of Mathematics, Stony Brook University, 100 Nicolls Road, 
Stony Brook, NY 11794}
\email{eunhye.lee@stonybrook.edu }
\begin{document}
 
 \begin{abstract}
  Enumerating integral orbits in prehomogeneous vector spaces plays an important role in arithmetic statistics.  We describe a method of proving subconvexity of the zeta function enumerating the integral orbits, illustrated by proving a subconvex estimate for the  Shintani $\zeta$ function enumerating class numbers of binary cubic forms.  
 \end{abstract}

 \thanks{This material is based upon work supported by the National Science
Foundation under agreement DMS-1802336. Any opinions, findings and
conclusions or recommendations expressed in this material are those of the
authors and do not necessarily reflect the views of the National Science
Foundation.}

\thanks{Robert Hough is supported by an Alfred P. Sloan Foundation Research 
Fellowship and a Stony Brook Trustees Faculty Award}
\maketitle

\section{Introduction}
The subconvexity problem is one of the important problems in the theory of zeta and $L$-functions.  The problem asks for a power saving estimate for the zeta or $L$-function on the critical line, compared to the bound obtained by interpolation between the regions of absolute convergence of the Dirichlet series using the functional equation. The end goal of the subconvexity problem is the Lindel\"{o}f Hypothesis, a consequence of the Riemann Hypothesis which has powerful analytic applications, see \cite{M18} and references therein.  An important class of zeta functions in modern analytic number theory are the zeta functions developed by M. Sato and Shintani enumerating integral orbits in prehomogeneous vector spaces ordered by invariants.  These include zeta functions important in arithmetic statistics enumerating low rank rings \cite{WY92}, \cite{Y93}, Epstein zeta functions and functions enumerating representation numbers of rational and irrational quadratic forms, multiple Dirichlet series enumerating several invariants \cite{S82a} and Eisenstein series of Selberg and others \cite{S82b}, see, for instance \cite{B05}, \cite{B10}, \cite{BST13}, \cite{BG14}, \cite{BH16}, \cite{BS15a}, \cite{BS15b}, \cite{BV15} for applications of orbit counting to arithmetic statistics. The purpose of this article is to introduce a new method of proving subconvexity of prehomogeneous vector space zeta functions, illustrated by proving $t$-aspect subconvexity for the Shintani zeta function enumerating class numbers of binary cubic forms, thus solving a problem of Thorne \cite{T14b}.  Moreover, the method is general, and we plan to return to prove subconvexity estimates for a general class of prehomogeneous vector space zeta functions in $t$ and $q$ aspects and for some zeta functions with an automorphic twist \cite{S94}, \cite{L18}, \cite{H17}.

 Let $V_\zed = \{f(x,y) = ax^3 + bx^2 y + cxy^2 + dy^3: a,b,c,d \in \zed\}$ be the space of integral binary cubic forms. The group $\SL_2(\zed)$ acts by integral change of variable. Shintani introduced zeta functions \cite{S72}
 \begin{equation}
  \xi^{\pm}(s) := \sum_{\substack{f \in \SL_2(\zed)\backslash V_\zed\\ \pm \Disc(f) >0}} \frac{1}{|\Stab(f)|} \frac{1}{|\Disc(f)|^s}, \qquad \RE(s)>1.
 \end{equation}
In their study of the adelization of the zeta functions Datskovsky-Wright \cite{DW86} and Ohno \cite{O97} (see Thorne \cite{T10}) describe the diagonalization 
\begin{equation}
 \xi^{\add}(s)= 3^{\frac{1}{2}}\xi^+(s) + \xi^-(s), \qquad \xi^{\sub}(s) = 3^{\frac{1}{2}}\xi^+(s) - \xi^-(s)
\end{equation}
and completed zeta functions 
\begin{align}
 \Lambda^{\add}(s) &= \left(\frac{432}{\pi^4}\right)^{\frac{s}{2}}\Gamma \left(\frac{s}{2}\right) \Gamma \left(\frac{s+1}{2}\right) \Gamma\left(\frac{s+\frac{1}{6}}{2} \right) \Gamma\left(\frac{s-\frac{1}{6}}{2}\right)\xi^{\add}(s),\\\notag
 \Lambda^{\sub}(s)&= \left(\frac{432}{\pi^4}\right)^{\frac{s}{2}}\Gamma \left(\frac{s}{2}\right) \Gamma \left(\frac{s+1}{2}\right) \Gamma\left(\frac{s+\frac{5}{6}}{2} \right) \Gamma\left(\frac{s+\frac{7}{6}}{2}\right)\xi^{\sub}(s),
\end{align}
which satisfy the self-dual functional equations $\Lambda(s) = \Lambda(1-s)$.  These zeta functions are degree 4, with analytic conductor $\sC(\frac{1}{2}+i\tau) = \tau^4$ as $\tau \to \infty$.  The convexity bound states $\left|\xi\left(\frac{1}{2} + i\tau\right)\right| \ll_\epsilon \tau^{1+\epsilon}$.  While it is known that these functions do not satisfy the Riemann Hypothesis \cite{T10}, it may be conjectured that they still satisfy the Lindel\"{o}f Hypothesis $\left|\xi\left(\frac{1}{2} + i\tau\right)\right| \ll_\epsilon \tau^\epsilon$.  In this direction we prove the following subconvexity estimate.
\begin{theorem}\label{untwisted_theorem}
 The Shintani zeta functions satisfy the sub-convex bound, for any $\epsilon > 0$, 
 \begin{equation}
  \xi^{\add}\left(\frac{1}{2} + i\tau\right), \xi^{\sub}\left(\frac{1}{2} + i\tau\right) \ll_\epsilon \tau^{\frac{98}{99}+\epsilon}
 \end{equation}
as $\tau \to \infty$.  
\end{theorem}
Let $\phi$ be a Hecke-eigen cusp form for $\SL_2(\zed)$.  The first author \cite{H17} introduced the automorphic twisted zeta functions 
\begin{equation}
 \sL^{\pm}(s,\phi) :=\sum_{\substack{f \in \SL_2(\zed)\backslash V_\zed\\ \pm \Disc(f) >0}} \frac{\phi(f)}{|\Stab(f)|} \frac{1}{|\Disc(f)|^s}.
\end{equation}
Our method is capable of proving the $t$-aspect subconvexity of these functions without significant modification, but we confine ourselves to the untwisted case in this paper for ease of presentation, see our preprint \cite{HL22} for a proof in the twisted case.

\subsection{Discussion of method}

As usual in the theory of zeta functions, the proof of Theorem \ref{untwisted_theorem} begins by expressing the zeta function in the critical strip via an approximate functional equation which expresses the zeta function as the sum of two Dirichlet polynomials, each of length square root of the conductor.  The theory of such functional equations for prehomogeneous vector space zeta functions was first developed by Sato and Shintani \cite{SS74}, with extensions to larger classes of zeta functions in works of Yukie \cite{Y93} and Saito \cite{S03}.  The coefficients in these Dirichlet polynomials are expressed as a sum over binary cubic forms restricted to a fundamental domain for the action of $\SL_2(\zed)\backslash \SL_2(\bR)$.  Bhargava's averaging trick \cite{B05} is used to average over many fundamental domains in order to bound the contribution of the lattice points in the cusp.  Having trimmed away the cusps, the points now lie essentially within a compact neighborhood of the identity; van der Corput's inequality is used to obtain cancellation in exponential sums in which the phase has controlled partial derivatives.

\subsection{Background and related work} The zeta function of a prehomogeneous vector space was introduced by Sato and Shintani \cite{SS74} in the 70s, with the case of binary cubic forms studied in detail by Shintani \cite{S72} who determined poles and residues along with the meromorphic continuation.  The class of functions have been extended over time.  F. Sato introduced series in multiple complex variables \cite{S82a}, \cite{S82b}, treating Eisenstein series associated to irrational indefinite quadratic forms.  Later he introduced a twisting automorphic form and showed a wide class of familiar number theoretic objects can be constructed in this way \cite{S94} including Langlands standard L-functions, and Dirichlet series of Maass constructed from systems of quadratic forms.  Datskovsky, Wright and Yukie developed the adelization with applications to low degree number fields \cite{DW86}, \cite{WY92}, \cite{Y93}.  Saito completed part of Yukie's work \cite{S03} by establishing criteria which guarantee the convergence and meromorphic continuation of the zeta functions.  Taniguchi and Thorne developed the local theory of the binary cubic form case more carefully \cite{TT13}, using this to prove a secondary main term in the Davenport-Heilbronn Theorem counting cubic number fields ordered by discriminant. Recently Wen-Wei Li has further developed the automorphic twisted zeta functions \cite{L18} proving local functional equations and giving an indication of the size of the class of objects considered.

The theory of subconvexity of zeta and L-functions is a very active current area of number theory.  After pioneering work of Duke, Friedlander and Iwaniec \cite{DFI93} subconvexity has frequently been established by the amplification method, which calculates a moment of the zeta or L-function in a family multiplied by a Dirichlet polynomial that selects for the function of interest.  Due to the high degree of the prehomogeneous vector space zeta functions, calculating moments is not quickly available and so the proof here uses a truncation method, together with an approach similar to the classical approach to the Riemann zeta function using van der Corput's method.  An interesting parallel result has been proved by Blomer recently for Epstein zeta functions \cite{B20}, which are prehomogeneous zeta functions.  In principle, similar bounds to Blomer's can be obtained from methods for PVS zeta functions, we intend to return to this issue in a forthcoming publication.

\section*{Notation}
We use the shorthand $e(x) = e^{2\pi ix}$, $c(x) = \cos(2\pi x)$, $s(x) = \sin(2\pi x)$. On $\bR/\zed$ we use the distance 
\begin{equation}
 \|x\|_{\bR/\zed} = \min_{n \in \zed} |x-n|.
\end{equation}
In $V_{\bR} = \{f(x,y)=ax^3 + bx^2y + cxy^2 + dy^3: a,b,c,d \in \bR\}$ define the infinity ball at $f$ of radius $R$ to be 
\begin{align}
B_R(f) = \{ &a'x^3 + b' x^2y + c' xy^2 + d'y^3: a',b',c',d' \in \bR,\\\notag& \max (|a-a'|,|b-b'|,|c-c'|,|d-d'|) \leq R\}.
\end{align}
 The averaging operator $\E_{s \in S} f(s)$ indicates $\frac{1}{|S|} \sum_{s \in S} f(s )$.  We use the following asymptotic notation.  For positive quantities $A, B$ which may depend on the parameter $\tau$, $A = O(B)$ means there is a constant $C>0$ such that $A \leq CB$. This has the same meaning as $A \ll B$.  We write $A \asymp B$ if $A \ll B \ll A$ and $A = o(B)$ if $\lim_{\tau \to \infty} \frac{A}{B} = 0.$ For a function $f$ on $\bR^+$, the Mellin transform is $\tilde{f}(s) = \int_0^\infty f(x)x^{s-1} dx$.  For differential operators with multi-indices $\alpha$, $D_\alpha = \partial_{x_1}^{\alpha_1}...\partial_{x_k}^{\alpha_k}$, $|\alpha| = \alpha_1 + ... + \alpha_k$.  We use the $C^j$ norms on $\bR^n$,
\begin{equation}
 \|f\|_{C^j}  = \sum_{|\alpha|\leq j}   \sup_{x \in \bR^n}\|D^\alpha f(x)\|.
\end{equation}

Our arguments use a smooth partition of unity on the positive reals.  Let $\sigma \geq 0$ be smooth and supported in $\left[\frac{1}{2}, 2\right]$ and satisfy $\sum_{n \in \zed} \sigma(2^n x) \equiv 1$ for $x \in \bR^+$.  

\section{Background}
As in the usual treatment of the subconvexity problem, we start  from an approximate functional equation. 

\subsection{The approximate functional equation}
Iwaniec and Kowalski \cite{IK04} Chapter 5 outlines a general framework for representing zeta functions with a functional equation inside the critical strip, called the approximate functional equation.  This framework applies to the Shintani zeta functions with a slight modification due to the pole at $\frac{5}{6}$ of $\xi^{\add}$.  Inside the domain of absolute convergence, write 
\begin{equation}
 \xi^{\add}(s) = \sum_n \frac{a^{\add}(n)}{n^s}, \qquad \xi^{\sub}(s) = \sum_n \frac{a^{\sub}(n)}{n^s}.
\end{equation}

The general framework takes the following data,
\begin{itemize}
 \item Gamma factor $\gamma(s) = \pi^{-\frac{ds}{2}} \prod_{j=1}^d \Gamma\left(\frac{s+\kappa_j}{2}\right), \qquad \kappa_j \geq -1$.
 \item Conductor $q$.
 \item Sign of functional equation $\epsilon$.
 \item Analytic conductor $q(s) = q \prod_{j=1}^d (|s+\kappa_j| + 3)$.
 \item Completed zeta function $\Lambda(s) = q^{\frac{s}{2}}\gamma(s) \zeta(s)$.  
 \item Functional equation $\Lambda(s) = \epsilon \Lambda(1-s)$.
\end{itemize}
The zeta functions $\xi^{\add}$ and $\xi^{\sub}$ fit into these frameworks with $d = 4$, $\epsilon = 1$, conductor $q = 432$ and gamma factors 
\begin{align}
 \gamma^{\add}(s) &= \pi^{-2s}\Gamma\left(\frac{s}{2} \right)\Gamma\left(\frac{s+1}{2} \right) \Gamma\left( \frac{s + \frac{1}{6}}{2}\right) \Gamma\left(\frac{s -\frac{1}{6}}{2} \right)  \\\notag
 \gamma^{\sub}(s) &= \pi^{-2s}\Gamma\left(\frac{s}{2} \right)\Gamma\left(\frac{s+1}{2} \right) \Gamma\left(\frac{s + \frac{5}{6}}{2} \right) \Gamma\left(\frac{s + \frac{7}{6}}{2} \right).
\end{align}

The proof of \cite{IK04} Theorem 5.3 yields the following representation.
\begin{theorem}\label{afe_theorem}
 Let $G(u)$ be any function which is holomorphic and bounded in $|\RE(u)| < 4$, even, with $G(0)=1$. For $0 < \RE(s) < 1$,
 \begin{align}
  \xi^{\add}(s) &= \sum_n \frac{a^{\add}(n)}{n^s}V_{s}^{\add}\left(\frac{n}{\sqrt{432}} \right) + \epsilon^{\add}(s)\sum_n \frac{a^{\add}(n)}{n^{1-s}}V_{ 1-s}^{\add}\left(\frac{n}{\sqrt{432}}\right) + R^{\add}(s)\\ \notag
  \xi^{\sub}(s) &= \sum_n \frac{a^{\sub}(n)}{n^s}V_{s}^{\sub}\left(\frac{n}{\sqrt{432}} \right) + \epsilon^{\sub}(s)\sum_n \frac{a^{\sub}(n)}{n^{1-s}}V_{ 1-s}^{\sub}\left(\frac{n}{\sqrt{432}}\right) + R^{\sub}(s)
 \end{align}
where
$\epsilon^*(s) = 432^{\frac{1}{2}-s} \frac{\gamma^*(1-s)}{\gamma^*(s)}$, 
\begin{equation}
 V^*_s(y) = \frac{1}{2\pi i} \int_{\RE u = 3} y^{-u}G(u)\frac{\gamma^*(s+u)}{\gamma^*(s)} \frac{du}{u}
\end{equation}
and 
\begin{align}
 R^{\add}(s) &= \left(\Res_{u=1-s}+\Res_{u = \frac{5}{6}-s} + \Res_{u = \frac{1}{6}-s} + \Res_{u = -s} \right)\frac{\Lambda^{\add}(s+u)}{432^{\frac{s}{2}}\gamma^{\add}(s)}\frac{G(u)}{u}\\ \notag
 R^{\sub}(s) &= \left(\Res_{u=1-s}+ \Res_{u = -s} \right)\frac{\Lambda^{\sub}(s+u)}{432^{\frac{s}{2}}\gamma^{\sub}(s)}\frac{G(u)}{u}.
\end{align}

\end{theorem}
For the choice of test function $G(u) = \left(\cos \frac{\pi u}{4A} \right)^{-4dA}$, \cite{IK04} Lemma 5.4 states the following.
\begin{lemma}\label{V_lemma}
 Suppose $\RE(s + \kappa_j) \geq 3\alpha >0$ for $1 \leq j \leq d$. Then the derivatives $V_s(y)$ satisfy 
 \begin{align}
  y^a V_s^{(a)}(y) &\ll \left(1 + \frac{y}{\tau^2}\right)^{-A}, \qquad 
  y^a V_s^{(a)}(y) = \delta_a + O\left(\left(\frac{y}{\tau^2}\right)^\alpha \right)
 \end{align}
where $\delta_0 = 1$, $\delta_a = 0$ if $a>0$, and the implied constants depend only on $\alpha, a, A$ and $d$.
\end{lemma}

The Lemma can be applied for either $V^{\add}_s$ and $V^{\sub}_s$ with $\RE(s) = \frac{1}{2}$ and $\alpha = \frac{1}{9}$.  Note that, with this choice of test function, the residue terms are $o(1)$ as $t \to \infty$ by Stirling's approximation, so may be ignored. Also, $|\epsilon^*(s)|=1$.

By forming linear combinations, to prove Theorem \ref{untwisted_theorem} it will suffice to prove a pair of estimates.
\begin{proposition}\label{afe_proposition}
 We have the pair of estimates
 \begin{align}
 \sum_{f \in \Gamma \backslash V_+} \frac{1}{|\Stab(f)|} \frac{1}{\Disc(f)^{\frac{1}{2}+i\tau}} V_{\frac{1}{2} + i\tau}\left(\frac{\Disc(f)}{\sqrt{432}}  \right) &\ll_\epsilon \tau^{\frac{98}{99}+\epsilon}\\ \notag
 \sum_{f \in \Gamma \backslash V_-}  \frac{1}{|\Disc(f)|^{\frac{1}{2}+i\tau}} V_{\frac{1}{2} + i\tau}\left(\frac{|\Disc(f)|}{\sqrt{432}}  \right) &\ll_\epsilon \tau^{\frac{98}{99}+\epsilon}.
\end{align}
\end{proposition}

\subsection{Background regarding $\GL_2(\bR)$}

We use the conventions of \cite{BST13} regarding Lie groups and Bhargava's averaging trick.  Let $\Gamma = \SL_2(\zed)$, and
\begin{align}
G^+ &= \{g \in \GL_2(\bR): \det(g)>0\}\\ \notag K &= \SO_2(\bR), \qquad k_\theta = \begin{pmatrix} c(\theta) & s(\theta)\\ -s(\theta) & c(\theta) \end{pmatrix}\\ \notag
 A_+ &= \left\{ a_t: t \in \bR_+\right\}, \qquad a_t = \begin{pmatrix} \frac{1}{t} &0\\ 0 & t \end{pmatrix}\\ \notag
 N &= \left\{n_u: u \in \bR\right\}, \qquad n_u = \begin{pmatrix} 1 &0\\u&1 \end{pmatrix}\\ \notag
 \Lambda &= \left\{d_\lambda: \lambda \in \bR_+\right\}, \qquad d_\lambda = \begin{pmatrix} \lambda &0\\0&\lambda\end{pmatrix}.
\end{align}

The Iwasawa decomposition of $\SL_2(\bR)$ expresses $g = n_u a_t k_\theta$.    Then for $g \in G^+$,$g = n_u a_t k_\theta d_\lambda$ and Haar measure on $G^+$ is given by $dg = du \frac{dt}{t^3} d\theta \frac{d\lambda}{\lambda}$.  Let $\sF$ denote the standard fundamental domain for $\SL_2(\zed)\backslash \SL_2(\bR)$,
\begin{align}
 \sF &= \{n_ua_tk_\theta: n_u \in N'(a), a_t\in A', k_\theta \in K\}\\ \notag
 A' &= \left\{ \begin{pmatrix} \frac{1}{t} &0\\0 & t \end{pmatrix}: t \geq \frac{3^{\frac{1}{4}}}{\sqrt{2}}\right\}\\ \notag
 N'(a) &= \left\{ \begin{pmatrix} 1 &0\\u & 1\end{pmatrix}: u\in \nu(a)\right\}
\end{align}
where $\nu(a)$ is the union of two subintervals of $\left[-\frac{1}{2}, \frac{1}{2}\right]$ and is the whole interval if $a \geq 1$.  For constants $A, B>0$, the \emph{Siegel set} $\sS(A,B)$ is
\begin{equation}
 \sS(A,B) = \left\{n_ua_tk_\theta: |u| \leq A, t \geq B, k_\theta \in K\right\}.
\end{equation}

We assume that $F$ is a smooth function, right $K$ invariant, supported on a Siegel set $\sS(A,B)$ with bounded derivatives and such that $\sum_{\gamma \in \Gamma} F(\gamma g) = 1$.  This may be constructed by letting $F_0(n_u a_t k_\theta) = f(u)g(t)$ where $f, g \geq 0$ are smooth functions, $f$ is supported in $|u| \leq A$, $g$ in $t \geq B$ with $g \equiv 1$ for $g \geq B+1$ and $f(u)g(t) > 0$ on a fundamental domain $\sF$.  Then $F(g) = \frac{F_0(g)}{\sum_{\gamma \in \Gamma} F_0(\gamma g)}$.  Notice that for all $t$ sufficiently large, only  $\gamma\in N$ satisfy $F_0(\gamma g) \neq 0$, since the Siegel set meets only finitely many fundamental domains, and those that are not unipotent translates are covered by $g$ for which $t$ is bounded.  This implies that the derivative in $t$ of $F$ vanishes in the cusp.   Extend $F$ to $G^+$ by $F(d_\lambda g) = F(g)$ for all $\lambda \in \bR^+$.

\subsection{The prehomogeneous vector space of binary cubic forms}
The space of real binary cubic forms is 
\begin{equation}
 V_{\bR} = \{f(x,y) = ax^3 + bx^2y + cxy^2 + dy^3: a,b,c,d \in \bR\}
\end{equation}
with integral forms $V_\zed$ having $a,b,c,d \in \zed$.  The discriminant is 
\begin{equation}
 \Disc(f) = b^2c^2 - 4ac^3 -4b^3d-27a^2d^2 + 18abcd.
\end{equation}
There is a bilinear pairing on $V_{\bR}$ which identifies it with its dual space,
\begin{equation}
 \langle f, g\rangle = f_1 g_4 - \frac{1}{3} f_2 g_3 + \frac{1}{3} f_3g_2 -f_4g_1.
\end{equation}
The Fourier transform on $V_\bR$ is given by $\hat{F}(\xi) = \int_{V_{\bR}} F(x) e^{-2\pi i \langle x, \xi\rangle} dx$.

The space has a left $\GL_2(\bR)$ action,
\begin{equation}
 \gamma \cdot f(x,y) = \frac{f((x,y)\gamma)}{|\det \gamma|},
\end{equation}
and $\Disc(\gamma \cdot f) = \det(\gamma)^2 \Disc(f)$.
Under this action there are two open orbits $V_{\pm} = \{f: \pm \Disc(f) > 0\}$ and a singular set $S = \{f: \Disc(f) = 0\}$.
The spaces have base points $f_{\pm}$,
\begin{equation}
 f_+ = \frac{1}{(108)^{\frac{1}{4}}} \left(0, 3, 0, -1\right), \qquad f_- = \frac{1}{\sqrt{2}}\left(0,1,0,1\right).
\end{equation}
Both $V_{\pm}$ can be identified as homogeneous spaces for $G^+$.  The mappings 
\begin{align}
 V_+ &= \left\{n_u a_t k_\theta d_\lambda \cdot f_+: u \in \bR, t\in \bR^+, \theta \in \left[0, \frac{1}{3}\right), \lambda \in \bR^+\right\},\\ \notag
 V_- &= \left\{n_u a_t k_\theta d_\lambda \cdot f_-: u \in \bR, t\in \bR^+, \theta \in \left[0, 1\right), \lambda \in \bR^+\right\}
\end{align}
are bijections between $V_{\pm}$ and subsets of $G^+$.  The stabilizer of $f_-$ is trivial and the stabilizer of $f_+$ is the rotation group generated by rotation by $\frac{2\pi}{3}$. 
The bilinear pairing satisfies $\langle x, y \rangle = \langle g\cdot x, g^\iota \cdot y\rangle$.

The rotation $k_\theta$ maps
\begin{align}
 k_\theta \cdot f_- &= \frac{1}{\sqrt{2}}\left(s(\theta), c(\theta), s(\theta), c(\theta) \right)\\ \notag
 k_\theta \cdot f_+ &= \frac{1}{(108)^{\frac{1}{4}}}\left(s(3\theta), 3c(3\theta), -3 s(3\theta),  -c(3\theta)\right).
 \end{align}
Meanwhile $a_t \cdot (a,b,c,d) = \left(\frac{a}{t^3}, \frac{b}{t}, tc, t^3d\right)$ and $n_u\cdot  (a,b,c,d) = (a,  3au+b,  3au^2+2bu + c, au^3 + bu^2 + cu + d)$, $d_\lambda \cdot (a,b,c,d) = (\lambda a, \lambda b, \lambda c, \lambda d)$. 

Putting these formulas together gives the change of coordinates to homogeneous coordinates.
\begin{align}
 n_u a_t k_\theta d_\lambda \cdot f_- &= \frac{\lambda}{\sqrt{2}} \Bigl(t^{-3} s(\theta), 3t^{-3}s(\theta)u + t^{-1}c(\theta), 3t^{-3}s(\theta)u^2 + 2t^{-1}c(\theta)u+ts(\theta), \\&\notag \qquad t^{-3}s(\theta)u^3 + t^{-1}c(\theta)u^2 + ts(\theta)u + t^3 c(\theta) \Bigr)\\
 \notag n_u a_t k_\theta d_\lambda \cdot f_+&= \frac{\lambda}{(108)^{\frac{1}{4}}}\Bigl(t^{-3} s(3\theta), 3t^{-3}s(\theta)u + 3t^{-1}c(3\theta), 3t^{-3}s(\theta)u^2 + 6t^{-1}c(3\theta)u-3ts(3\theta), \\&\notag \qquad t^{-3}s(3\theta)u^3 + 3t^{-1}c(3\theta)u^2 - 3ts(3\theta)u - t^3 c(3\theta) \Bigr).
\end{align}

Set $B = B(C) = \{f = (a,b,c,d) \in V_\bR: 3a^2 + b^2 + c^2 + 3d^2 \leq C, |\Disc(f)| \geq 1\}$, which is a set which is $K$-invariant.
Define 
\begin{equation}
 B_{\pm}(u,t,\lambda, X) = n_u a_t d_\lambda B \cap\{v \in V_{\pm}: |\Disc(v)| \leq X\}.
\end{equation}

Given $f \in V_+$, let $n_u a_t k_\theta d_\lambda \cdot f_+ = f$, $u \in \bR, t \in \bR^+, \theta \in \left[0, \frac{1}{3}\right), \lambda \in \bR^+$ and set $g_f = n_u a_t k_\theta d_\lambda$.  Given $f \in V_-$, let $n_u a_t k_\theta d_\lambda \cdot f_- = f$ with $u \in \bR, t \in \bR^+, \theta \in [0, 1), \lambda \in \bR^+$ and set $g_f = n_u a_t k_\theta d_\lambda$.  In these expressions let $u = u(f), t= t(f), \theta = \theta(f)$.  
\begin{lemma}
 Suppose for some $v \in B$ that $f = n_u a_t \cdot v$ then $\log t = \log t_f + O(1)$.
\end{lemma}
\begin{proof}
 Write $v = n_{u'}a_{t'} k_{\theta'}d_{\lambda'} \cdot f_{\pm}$.  Then $f = n_{u + t^2 u'} a_{tt'} k_{\theta'}d_{\lambda'} \cdot f_{\pm}$. Since $t(f) = tt'$, the lemma follows.
\end{proof}

Integrals over $V_\pm$ may be expressed, for $f \in C_0(V_{\pm})$,
\begin{align}
\label{integration_formulae} \int_{V_+} f(v) \frac{dv}{\Disc(v)} &= \int_{-\infty}^\infty \int_{0}^\infty \int_0^{\frac{1}{3}} \int_0^\infty f(n_u a_t k_\theta d_\lambda \cdot f_+) \frac{d\lambda}{\lambda} d\theta \frac{dt}{t^3} du\\ \notag
  \int_{V_-} f(v) \frac{dv}{|\Disc(v)|} &= \int_{-\infty}^\infty \int_{0}^\infty \int_0^{1} \int_0^\infty f(n_u a_t k_\theta d_\lambda \cdot f_-) \frac{d\lambda}{\lambda} d\theta \frac{dt}{t^3} du.
\end{align}
Thus $|da \wedge db \wedge dc \wedge dd| = \frac{\lambda^3}{t^3} |d\lambda \wedge d\theta \wedge dt \wedge du|$.

 Let $\omega$ be a smooth, non-negative $K$ invariant function supported in $B$.
Let 
\begin{align}
 M_+ &= \int_{B \cap V_+} \frac{\omega(v)dv}{\Disc(v)},\qquad
 M_- = \int_{B \cap V_-} \frac{\omega(v)dv}{|\Disc(v)|}.
\end{align}

The following lemma estimates the dependence in switching between rectangular and homogeneous coordinates.
\begin{lemma}\label{Jacobian_lemma}
 When $u, t, \theta$ vary in a Siegel set and $\lambda \geq 1$, and $v \in B$ the change of coordinates $(a,b,c,d)= n_u a_t k_\theta d_\lambda \cdot v$ satisfies
 \begin{align}
  \frac{\partial(a,b,c,d)}{\partial( u,t, \theta, \lambda)} &= \begin{pmatrix}0 & O(\lambda t^{-3})& O(\lambda t^{-1})& O(\lambda t) \\ O(\lambda t^{-4}) & O(\lambda t^{-2})& O(\lambda) & O(\lambda t^2)\\ O(\lambda t^{-3}) & O(\lambda t^{-1})& O(\lambda t)& O(\lambda t^3)\\
  O(t^{-3}) & O(t^{-1}) & O(t) & O(t^3)\end{pmatrix}\\ \notag
  \frac{\partial( u,t, \theta, \lambda)}{\partial(a,b,c,d)} &= \begin{pmatrix}O(\lambda^{-1}t^5) & O(\lambda^{-1}t^4)& O(\lambda^{-1}t^3)& O(t^3)\\ O(\lambda^{-1}t^3)&O(\lambda^{-1}t^2)&O(\lambda^{-1}t)& O(t)\\
  O(\lambda^{-1}t) & O(\lambda^{-1})& O(\lambda^{-1}t^{-1})& O(t^{-1}) 
  \\
  O(\lambda^{-1}t^{-1})& O(\lambda^{-1}t^{-2})& O(\lambda^{-1}t^{-3}) & O(t^{-3})\end{pmatrix}.
 \end{align}

\end{lemma}
\begin{proof}
 The first Jacobian follows directly from the expression in homogeneous coordinates.  The second follows from combining the estimate for the adjugate matrix with the factor of integration $\lambda^3 t^{-3}$.
\end{proof}

The proof of the theorems rely on estimates for the derivatives and logarithmic derivatives of the discriminant which are uniform in the cuspidal parameter $t$.

For a multi-index $\alpha = (\alpha_1, \alpha_2, \alpha_3, \alpha_4)$ let $|\alpha| = \alpha_1 + \alpha_2 + \alpha_3 + \alpha_4$ and let 
$D^\alpha = D_a^{\alpha_1}D_b^{\alpha_2}D_c^{\alpha_3}D_d^{\alpha_4}$.
\begin{lemma}\label{derivative_lemma}
Let $C>0$ be a constant. Let $f_0$ be such that $|\Disc(f_0)| \geq 1$ and $3a^2 + b^2 + c^2 +3d^2 \leq C$.  Let $f= n_u a_t d_\lambda \cdot f_0$ with $u = O(1)$ and $t, \lambda \gg 1$. Then, with implicit constants depending on $C$,
\begin{equation}
 D^\alpha \Disc(f) = O\left(\lambda^{4-|\alpha|}t^{3|\alpha|}\right)
\end{equation}
and
\begin{equation}
 D^\alpha \log |\Disc(f)| = O\left(\lambda^{-|\alpha|}t^{3|\alpha|}\right)
\end{equation}
while
\begin{equation}
 \max_{D \in \{D_a^3, D_b^3, D_c^3, D_d^3\}} |D\log |\Disc(f)|| \gg \frac{1}{t^9 \lambda^3}.
\end{equation}
\end{lemma}

\begin{proof}
Use $n_u a_t = a_t n_{\frac{u}{t^2}}$ to write $f = a_t n_{\frac{u}{t^2}} d_\lambda f_0$, and let $u_0 = \frac{u}{t^2} = O(1)$.  When $u_0$ is held fixed and $t$ and $\lambda$ vary,  $D^\alpha \Disc(f)$ is homogeneous in $\lambda$ and in $t$.  For instance, the discriminant itself is a function of $\lambda$ only, and each coordinate is homogeneous under the action of $a_t$.  Differentiating in a given coordinate changes the homogeneous degree in $\lambda$ and $t$.  The degree in $\lambda$ is $4-|\alpha|$ and the degree in $t$ ranges between $t^{-3|\alpha|}$ and $t^{3|\alpha|}$.  This proves the first bound.

We have $D^\alpha \log |\Disc(f)|$ is the sum of monomials, each of which is the product of terms of type $\frac{D^\beta \Disc(f)}{\Disc(f)}$, with the total degree of the monomial equal to $|\alpha|$.  The claimed bound now follows from the first part.

By the homogeneity in $t$ and $\lambda$, for the last part it suffices to check that if $f_1 = n_{u_0}f$ then 
\begin{equation}
 \max_{D\in \{D_a^3, D_b^3, D_c^3, D_d^3\}} |D \log |\Disc(f_1)|| \gg 1.
\end{equation}
The estimate now is taken over $f_1$ in a compact set, and the lower bound can be established, for instance, in Mathematica.

\end{proof}

\begin{lemma}\label{f_variation_lemma}
 Let $f = (a,b,c,d) \in \bR^4$ be a form with $\lambda_f \geq 1$ and $t_f \gg 1$.  For every constant $C_1 > 1$ there is a constant $C_2 > 0$ so that if $\|\tilde{f} - f\|_2 \leq C_2 \frac{\lambda_f}{t_f^3}$ then 
 \begin{equation}
  \frac{\lambda_f}{C_1} \leq \lambda_{\tilde{f}} \leq C_1 \lambda_f.
 \end{equation}
\end{lemma}
\begin{proof}
 We prove this for $\lambda_f^4$ and $\lambda_{\tilde{f}}^4$ instead, which is the magnitude of the discriminant, and hence given by a degree 4 polynomial in the coefficients of the form.  Let $R = \|\tilde{f} - f\|_2$.  Since the degree 4 Taylor expansion of the discriminant is exact, with $\alpha = \frac{R t_f^3}{\lambda_f} $, 
 \begin{equation}
  \lambda_{\tilde{f}}^4 - \lambda_f^4 = O\left(\lambda_f^4 (\alpha + \alpha^2 + \alpha^3 + \alpha^4)  \right)
 \end{equation}
 which suffices for the proof.
\end{proof}
A similar claim now holds for $t_f$.
\begin{lemma}\label{t_variation_lemma}
 Let $f = (a,b,c,d) \in \bR^4$ be a form with $\lambda_f \geq 1$ and $t_f \gg 1$.  For every constant $C_1 > 1$ there is a constant $C_2 > 0$ so that if $\|\tilde{f} - f\|_2 \leq C_2 \frac{\lambda_f}{t_f^3}$ then 
 \begin{equation}
  \frac{t_f}{C_1} \leq t_{\tilde{f}} \leq C_1 t_f.
 \end{equation}
 
\end{lemma}

\begin{proof}
 By applying the previous lemma, we may restrict to a ball about $f$ so that $\lambda$ is within constants of $\lambda_f$ throughout the ball.  Now integrate in one direction at a time the partial derivative bounds of Lemma \ref{Jacobian_lemma}.  For instance, the bound $\frac{\frac{\partial t}{\partial a}}{t^4} = O(\lambda^{-1})$ obtains the estimate
 \begin{equation}
  \left| \frac{t_f(a_1)^3}{t_f(a_2)^3} - 1\right| \leq \frac{c |a_2-a_1| t_f(a_1)^3}{\lambda}
 \end{equation}
which is of the correct shape.  The remaining coordinates are similar, but obtain better bounds.
\end{proof}

The previous lemmas permit the pointwise bound of Lemma \ref{Jacobian_lemma} to hold as a sup bound in balls of radius $\ll \frac{\lambda}{t^3}$ about a form $f$.

Let $h(n)$ denote the number of classes of integral binary cubic forms of discriminant $n$.  We recall an easy consequence of the Davenport-Heilbronn Theorem \cite{DH71},
\begin{equation}\label{average_order_h_n}
 \sum_{0 < |n| \leq X} h(n) \ll X
\end{equation}
as $X \to \infty$.

\subsection{Bhargava's averaging technique}
Bhargava's averaging trick counts lattice points by averaging over many fundamental domains.  This is useful in bounding the number of points appearing in a cusp, as the following lemmas show.

\begin{lemma}
 Let $u = O(1)$ and $t \gg 1$.  The number of lattice points $(a,b,c,d)$ in $B(u,t,\lambda, X)$ with $a \neq 0$ is 
 \begin{equation}
  \left\{ \begin{array}{lll} 0 && \frac{C\lambda}{t^3} <1\\ \vol(B(u,t,\lambda, X)) + O(\max(C^3t^3 \lambda^3, 1)) && \text{otherwise}\end{array}\right..
 \end{equation}
The number of lattice points with $a = 0$ is
\begin{equation}
 \left\{ \begin{array}{lll} 0 && \frac{C\lambda}{t} < 1\\ O(\max(C^3t^3\lambda^3,1)) && \text{otherwise}\end{array}\right..
\end{equation}

\end{lemma}
\begin{proof}
 The first part is Lemma 25 of \cite{BST13}.  For the second part, if $a = 0$ and $\Disc(f) \neq 0$ then $b \neq 0$, which imposes the constraint $\frac{C\lambda}{t} \geq 1$.  The number of choices for $b,c,d$ are now $O\left(\frac{C\lambda}{t}\right), O(C\lambda t)$ and $O(C\lambda t^3)$, which proves the remaining claim.
\end{proof}

\begin{lemma}\label{large_T_lemma}
 When $v$ is chosen at random from $B$ according to the probability measure proportional to $\frac{\omega(v)dv}{|\Disc(v)|}$, where $\omega$ is a smooth non-negative function, the expected number of lattice points of discriminant of size at most $X$ with $t(f) > T$ and $a \neq 0$ is $O\left(\frac{X}{T}\right) + O\left(X^{\frac{5}{6}}\right)$.  
 
 The expected number of lattice points of discriminant of size at most $X$ with $t(f) \leq T$ and $a = 0$ is $O\left(X^{\frac{3}{4}}T\right)$.
\end{lemma}

\begin{proof}
 This follows on integrating the bounds from the previous lemma together with the integration formulae in (\ref{integration_formulae}), keeping in mind $|\Disc(v)| = \lambda^4$ in the formulae.
\end{proof}

\subsection{Bounds for exponential sums}\label{exponential_sum_section}

Our theorems exhibit cancellation in exponential sums using van der Corput's inequality \cite{S04}, p.216. 
\begin{lemma}[van der Corput's inequality] Let $c_1, c_2, ..., c_N$ be complex numbers and let $1 \leq H <N$.  Then
 \begin{align} 
  \left|\sum_{n=1}^N c_n \right|^2 &\leq \frac{N+H}{H+1}\sum_{n=1}^N |c_n|^2\\\notag &+ \frac{2(N+H)}{H+1}\sum_{h=1}^H\left(1 - \frac{h}{H+1}\right) \left|\sum_{n=1}^{N-h} c_{n+h}\overline{c_n}\right|.
 \end{align}
\end{lemma}
\noindent 
We also use the simple bound for a linear phase, for $\alpha \in \bR\setminus \zed$,
\begin{equation}
 \sum_{j=1}^N e^{2\pi i \alpha j} \ll \min\left(N, \frac{1}{\|\alpha\|_{\bR/\zed}}\right).
\end{equation}

The following estimates for exponential sums are used in the proofs.  Let $|\Disc(f)| \asymp Y$, $t_f \ll T_1$ and $B_R(f)$ be the $\infty$-ball of radius $R$ about $f$.

\begin{lemma}\label{exponential_sum_lemma}
Let $Y\gg \tau^{\frac{4}{3}}T_1^{12}$ and $T_1 = o\left(Y^{\frac{1}{84}} \right)$.  With the choice $R = \frac{Y^{\frac{1}{4}}}{\tau^{\frac{7}{27}} T_1^{\frac{7}{3}}}$, we have the bound
 \begin{equation}
  \E_{y \in B_{R}}\left[|\Disc(f+y)|^{-i\tau}\right] \ll   \frac{T_1^{\frac{8}{3}}}{\tau^{\frac{1}{27}}}(\log \tau)^{\frac{4}{9}}.
 \end{equation}

\end{lemma}

\begin{proof}
 The proof uses the following parameters.
\begin{itemize}
 \item $R$, radius of averaging ball
 \item $R_1<R$, a spacing parameter necessary to obtain cancellation in exponential sums after differencing
 \item $N = 2 \frac{R}{R_1} + O(1)$, the length of the exponential sum over which we find cancellation
 \item $H_1 = O(N)$, first van der Corput parameter
 \item $H_2 = O(N)$, second van der Corput parameter
 \item $\alpha = \max_{D \in \{D_a^3, D_b^3, D_c^3, D_d^3\}} |D\tau \log|\Disc(f+y)|_{y=0}|$.  This satisfies
 \begin{equation}
  \frac{\tau}{T_1^9 Y^{\frac{3}{4}}} \ll \alpha \ll \frac{\tau T_1^9}{Y^{\frac{3}{4}}}.
 \end{equation}

 \item $\delta = \frac{\tau T_1^{12}}{Y}$, a bound up to constants for fourth derivatives $D^\beta \tau \log |\Disc(f+y)|$, $|\beta| = 4$.
\end{itemize}
The argument will require $\delta R^4 = o(1)$ and $\alpha R^3 \gg 1$, which imposes the constraint $T_1 = o(Y^{\frac{1}{84}})$.

 Taylor expand to degree 4 to express 
 \begin{align}
  &\frac{1}{|\Disc(f+y)|^{i\tau}}\\&\notag = \exp\left(-i\tau \log |\Disc(f)| - i\tau \sum_{0 < |\beta| \leq 3} \frac{D^\beta \log |\Disc(f)|}{\beta!}y^\beta + O\left(\delta R^4\right) \right).
 \end{align}
 Subject to the condition
 \begin{equation}
  \delta R^4 = o(1),
 \end{equation}
the error contributes $O\left(\delta R^4\right)$ to the average.

 By Lemma \ref{derivative_lemma}, one of $D_a^3, D_b^3, D_c^3, D_d^3$ applied to $\log |\Disc(f)|$ is $\gg \frac{1}{T_1^9 |\Disc(f)|^{\frac{3}{4}}}$, say without loss of generality that $D_a^3$ satisfies this bound. The third derivatives also satisfy the upper bound $\ll \frac{T_1^9}{|\Disc(f)|^{\frac{3}{4}}}$. Let $z = f+y$ and $F(z) = -\tau\log |\Disc(f)| - \tau \sum_{0 < |\alpha| \leq 3} \frac{D^\alpha \log |\Disc(f)|}{\alpha!}y^\alpha$ and let for some parameter $1 \leq R_1 \leq R$, $z_n = z+ (n R_1,0,0,0)$,
\begin{equation}
 S_z = \sum_n \one(z_n \in f + B_R) e^{iF(z_n)}
\end{equation}
which is a sum of length $N = 2 \frac{R}{R_1} + O(1)$ by the definition of $z_n$ and the support of the indicator function. 

We bound $S_z$ by applying van der Corput's inequality twice.  Let $H_1, H_2 \ll N$ be parameters, and define $\Delta_h G(n) = G(n+h)-G(n)$.  Applying van der Corput once
\begin{align}
 |S_z|^2 \leq &\frac{(N+H_1)N}{H_1}  \\ \notag&+\frac{2(N+H_1)}{H_1+1}\sum_{h_1=1}^{H_1}\left(1- \frac{h_1}{H_1+1} \right)\left|\sum_{n} e^{-i \Delta_{h_1} F(z_n)}\one\left(z_n, z_{n+h_1} \in f + B_R \right)\right|.
\end{align}
Let $S_{z,h_1} = \sum_{n} e^{-i \Delta_{h_1} F(z_n)}\one\left(z_n, z_{n+h_1} \in f + B_R \right)$ and bound
\begin{align}
 |S_{z,h_1}|^2 &\leq \frac{(N -h_1 + H_2)N}{H_2}\\ \notag&+ \frac{2(N-h_1 + H_2)}{H_2+1} \sum_{h_2 = 1}^{H_2} \left(1 - \frac{h_2}{H_2+1}\right) \left|\sum_{n}e^{-i \Delta_{h_1}\Delta_{h_2}F(z_n)}\one\left(z_n,  z_{n+h_1+h_2} \in f + B_R\right) \right|.
\end{align}
The phase in the exponent is now a linear function of $n$ with leading coefficient of order
$\alpha R_1^3 h_1 h_2$.
Make the constraint that this is $\leq \frac{1}{2}$.
Let 
\begin{equation}S_{z, h_1, h_2} = \sum_{n}e^{-i \Delta_{h_1}\Delta_{h_2}F(z_n)}\one\left(z_n,  z_{n+h_1+h_2} \in f + B_R\right).\end{equation}  Thus
\begin{equation}
 |S_{z,h_1, h_2}| \ll \frac{1}{\alpha R_1^3 h_1 h_2 }.
\end{equation}
Combining the above estimates with Cauchy-Schwarz obtains
\begin{align}
 |S_z|^4 &\ll \frac{N^4}{H_1^2} + \frac{N^2}{H_1^2}\left(H_1 \sum_{h_1=1}^{H_1} |S_{z,h_1}|^2 \right)\\
 &\notag \ll \frac{N^4}{H_1^2} + \frac{N^2}{H_1}\left( \sum_{h_1=1}^{H_1} \left(\frac{N^2}{H_2} + \frac{N}{H_2} \sum_{h_2=1}^{H_2} |S_{z, h_1, h_2}|\right) \right)\\
 &\notag \ll \frac{N^4}{H_1^2} + \frac{N^4}{H_2} + \frac{N^3}{H_1H_2} \sum_{h_1=1}^{H_1}\sum_{h_2=1}^{H_2} \frac{1}{\alpha R_1^3 h_1 h_2 }.
\end{align}
It follows that 
\begin{align}
 \frac{|S_z|}{N} \ll\frac{1}{H_1^{\frac{1}{2}}} + \frac{1}{H_2^{\frac{1}{4}}} + \frac{1}{N^{\frac{1}{4}}}\left(\frac{\log H_1 \log H_2}{H_1H_2} \right)^{\frac{1}{4}} \left(\frac{1}{\alpha R_1^3} \right)^{\frac{1}{4}}.
 \end{align}
 Now choose $H_2 = H_1^2$, $H_1H_2 = H_1^3 = \frac{c}{\alpha R_1^3}$, and impose the restriction
$
 H_2  \asymp N,
$
to obtain
 \begin{align}
 \frac{|S_z|}{N}
 &\ll \frac{1}{H_1^{\frac{1}{2}}} + \frac{1}{H_2^{\frac{1}{4}}} + \frac{(\log H_1 \log H_2)^{\frac{1}{4}}}{N^{\frac{1}{4}}}.
\end{align}
This leads to $\left(\frac{R}{R_1} \right)^{\frac{3}{2}} R_1^3 \alpha \asymp 1$ or $R_1 \asymp \frac{1}{R \alpha^{\frac{2}{3}}}$.  Thus $N = R^2 \alpha^{\frac{2}{3}}$. Note that $N$ is bounded by a power of $\tau$ so that $\log H_1, \log H_2 \ll \log \tau$.  Combining our estimates thus far, and averaging over $z$
\begin{equation}
 \E_{y \in B_R} \left[\frac{1}{|\Disc(f+y)|^{i\tau}}\right] \ll \delta R^4 + \frac{(\log \tau)^{\frac{1}{2}}}{\alpha^{\frac{1}{6}} R^{\frac{1}{2}}}.
\end{equation}
Now apply $\delta = \frac{\tau T_1^{12}}{Y}$, $\alpha \gg \frac{\tau}{T_1^9 Y^{\frac{3}{4}}}$ to bound 
\begin{equation}
 \E_{y \in B_R} \left[\frac{1}{|\Disc(f+y)|^{i\tau}}\right] \ll \frac{\tau T_1^{12}}{Y} R^4 + \frac{T_1^{\frac{3}{2}}Y^{\frac{1}{8}}}{\tau^{\frac{1}{6}} R^{\frac{1}{2}}}(\log \tau)^{\frac{1}{2}}.
\end{equation}

Setting the two error terms equal leads to 
\begin{equation}
 R = \frac{Y^{\frac{1}{4}}}{\tau^{\frac{7}{27}} T_1^{\frac{7}{3}}}(\log \tau)^{\frac{1}{9}}
\end{equation}
and the bound $ \ll \frac{T_1^{\frac{8}{3}}}{\tau^{\frac{1}{27}}}(\log \tau)^{\frac{4}{9}}$.  

\end{proof}

For forms $f$ with $a=0$, set $f_d = f + (0,0,0,d)$.  For $|\Disc(f)| \asymp Y$ and $t_f \gg T_2$ we have the following estimate.  
\begin{lemma}\label{exponential_sum_reducibles_lemma}
 Assume that $\tau$ is bounded by a sufficiently small constant times $T_2^3 Y^{\frac{1}{4}}$.  With the choice $R = \frac{T_2^2 Y^{\frac{1}{2}}}{\tau^{\frac{2}{3}}}$, we have the bound
 \begin{equation}
  \E_{|d|\leq R}\left[|\Disc(f_d)|^{-i\tau} \right] \ll \frac{Y^{\frac{1}{2}}}{\tau^{\frac{1}{3}}T_2^2}.
 \end{equation}

\end{lemma}
\begin{proof} We may assume that $\frac{Y^{\frac{1}{2}}}{\tau^{\frac{1}{3}} T_2^2} \ll 1$ since otherwise the claim is trivial.

Recall that at $a = 0$, $\Disc(f) = b^2c^2 - 4b^3 d$. Thus $\frac{\partial}{\partial d} \log |\Disc(f)| \asymp \frac{b^3}{Y} .$ Notice that by the homogeneity in $\lambda$ and $t$, $1 \ll b \frac{Y^{\frac{1}{4}}}{T_2}$. Also, 
 \begin{equation}
  \frac{\partial^2}{\partial d^2} \log |\Disc(f)| = O\left(\frac{b^6}{Y^2} \right) = O\left(\frac{1}{T_2^6 Y^{\frac{1}{2}}} \right).
 \end{equation}
Impose the constraint $R^2 = o\left(\frac{T_2^6 Y^{\frac{1}{2}}}{\tau} \right)$ and Taylor expand the exponent to degree 2 in the average.  With $\alpha = \frac{\partial}{\partial d} \log |\Disc(f)|$, $\frac{1}{Y} \ll \alpha \ll \frac{1}{T_2^3 Y^{\frac{1}{4}}}$,
\begin{align}
 \left|\E_{|d|\leq R}\left[|\Disc(f_d)|^{-i\tau} \right] \right| = \left| \E_{|d|\leq R} e^{-i\tau \alpha d}\right| + O\left(\frac{\tau R^2}{T_2^6 Y^{\frac{1}{2}}} \right).
 \end{align}
 Since $\alpha \tau \ll 1$ by the constraint of the Lemma, using the bound for the sum of a linear phase we obtain 
 \begin{align}  \left|\E_{|d|\leq R}\left[|\Disc(f_d)|^{-i\tau} \right] \right|\ll  \frac{1}{R\alpha \tau} + \frac{\tau R^2}{T_2^6 Y^{\frac{1}{2}}}. 
\end{align}

Choose optimally, $R^3 = \frac{T_2^6 Y^{\frac{1}{2}}}{\alpha \tau^2}$.  The worst bound occurs by minimizing $\alpha$, which obtains $R = \frac{T_2^2 Y^{\frac{1}{2}}}{\tau^{\frac{2}{3}}}$ and an estimate for the exponential sum of
$\ll \frac{Y^{\frac{1}{2}}}{\tau^{\frac{1}{3}}T_2^2}$.  
\end{proof}

\section{Proof of Theorem \ref{untwisted_theorem}}

We begin the proof of Proposition \ref{afe_proposition} by introducing the averaging technique of \cite{BST13}. Let $n_{-} = 1$, $n_+ =3$ be the multiplicity with which $G^+ \cdot f_{\pm}$ covers $V_{\pm}$.  Let $F$ be the smooth partition of unity function on $\SL_2(\bR)$, supported on a Siegel set, right $K$ invariant, such that $\sum_{\gamma \in \Gamma} F(\gamma g) = 1$.  Let $H_{\pm}$ be a maximal subset of $G^+$ so that $H_\pm \cdot f_{\pm} = B \cap V_{\pm}$.  Define
\begin{equation}
 \Sigma =\sum_{f \in \Gamma \backslash V_{\pm}} \frac{1}{|\Stab(f)|} \frac{1}{|\Disc(f)|^{\frac{1}{2}+i\tau}} V_{\frac{1}{2} + i\tau}\left(\frac{|\Disc(f)|}{\sqrt{432}}  \right).
\end{equation}

Then as in \cite{BST13}, p.458 eqn. (19), 
\begin{align}
 \Sigma&= \frac{1}{n_{\pm}M_{\pm}} \int_{v \in B \cap V_{\pm}} \frac{\omega(v)dv}{|\Disc(v)|}\sum_{f \in V_\zed\cap V_{\pm}} \frac{V_{\frac{1}{2} + i\tau}\left(\frac{|\Disc(f)|}{\sqrt{432}}  \right)}{|\Disc(f)|^{\frac{1}{2}+i\tau}} \sum_{g \cdot v = f} F(g)
  \\ \notag
 &= \frac{1}{n_{\pm}M_{\pm}} \sum_{f \in V_\zed\cap V_{\pm}} \frac{V_{\frac{1}{2} + i\tau}\left(\frac{|\Disc(f)|}{\sqrt{432}}  \right)}{|\Disc(f)|^{\frac{1}{2}+i\tau}}\int_{g \in \SL_2(\bR)} \int_0^\infty \sum_{h \in H_{\pm}: f = d_\lambda gh \cdot f_{\pm}}\omega(h\cdot f_{\pm})F(g) dg \frac{d\lambda}{\lambda}
\end{align}
where $M_{\pm} = \int_{B \cap V_{\pm}} \frac{\omega(v)dv}{|\Disc(v)|}$.

Next introduce a smooth partition of unity to control the size of the discriminant.  Let $\sigma \in C_c^\infty(\bR^+)$ satisfy $\sigma \geq 0$ and 
\begin{equation}
 \sum_{n \geq A} \sigma\left(\frac{x}{2^n}\right) = 1, \qquad x \geq 1.
\end{equation}
Thus 
\begin{align}
 \Sigma=&\frac{1}{n_{\pm}M_{\pm}} \sum_{f \in V_\zed\cap V_{\pm}}\sum_{n \geq A} \sigma\left(\frac{|\Disc(f)|}{2^n}\right) \frac{V_{\frac{1}{2} + i\tau}\left(\frac{|\Disc(f)|}{\sqrt{432}}  \right)}{|\Disc(f)|^{\frac{1}{2}+i\tau}}\\ \notag &\times\int_{g \in \SL_2(\bR)} \int_0^\infty \sum_{h \in H_{\pm}: f = d_\lambda gh \cdot f_{\pm}}\omega(h \cdot f_{\pm})F(g) dg \frac{d\lambda}{\lambda}.
\end{align}
We may assume $Y > \tau^{2 - \frac{2}{99}}$, bounding the initial part of the sum trivially, with acceptable error.

Next truncate in $t$.  Let $T_1 = T_1(Y)$ and $T_2 = T_2(Y)$ be parameters.  Recall that $f = n_u a_t k_\theta \cdot f_{\pm}$.  By Lemma \ref{large_T_lemma} the expected number of classes of forms $f$ with $a \neq 0$ and $t > T_1$, $|\Disc(f)| \asymp Y$ is $O\left(\frac{Y}{T_1} + Y^{\frac{5}{6}} \right)$, while the number of forms with $a= 0$ and $t < T_2$ is $O\left(Y^{\frac{3}{4}} T_2 \right)$.  Let $\Sigma_1(Y)$ indicate the sum over forms with $a \neq 0$ and $t \leq T_1$ and $\Sigma_2(Y)$ the sum over $a = 0$ with $t \geq T_2$.  Thus
\begin{equation}
 \Sigma(Y) = \Sigma_1(Y) + \Sigma_2(Y) +  O\left(\left(\frac{\sqrt{Y}}{T_1} + Y^{\frac{1}{3}} + Y^{\frac{1}{4}}T_2\right)\left(1 + \frac{Y}{\tau^2}\right)^{-A} \right).
\end{equation}

In $\Sigma_1(Y)$, introduce a space average.   Let $R = R(Y)$ be a parameter satisfying $R \ll \frac{Y^{\frac{1}{4}}}{T_1^3}$ and let $B_R = \{y \in V_\zed: \|y\|_\infty \leq R\}$.  Note that by Lemmas \ref{f_variation_lemma} and \ref{t_variation_lemma}, for $y \in B_R$, $|\Disc(f+y)| \asymp |\Disc(f)|$ and $t_{f+y} \asymp t_f$.  The averaged sum is, with $'$ indicating $a \neq 0$ and $t \leq T_1$,
\begin{align}
  \Sigma_1'(Y)=&\frac{1}{n_{\pm}M_{\pm}} {\sum_{f \in V_\zed \cap V_{\pm}}}'\E_{y \in B_{R(Y)}} \sigma\left(\frac{|\Disc(f+y)|}{Y}\right) \frac{V_{\frac{1}{2} + i\tau}\left(\frac{|\Disc(f+y)|}{\sqrt{432}}  \right)}{|\Disc(f+y)|^{\frac{1}{2}+i\tau}}\\& \notag\times\int_{g \in \SL_2(\bR)} \int_0^\infty \sum_{h \in H_{\pm}: f+y = d_\lambda gh \cdot f_{\pm}} \omega(h \cdot f_{\pm})F(g) dg \frac{d\lambda}{\lambda}.
\end{align}
Notice that in $\Sigma_1'(Y)$, the sum over $f$ is restricted to forms $f$ with $t \leq T_1$, which is different than the condition $t_{f+y} \leq T_1$.  This introduces a difference between $\Sigma_1(Y)$ and $\Sigma_1'(Y)$. 
We have, 
$|\Sigma_1(Y) - \Sigma_1'(Y)| =  O\left(\left(\frac{\sqrt{Y}}{T_1} + Y^{\frac{1}{3}}\right)\left(1+\frac{Y}{\tau^2} \right)^{-A}  \right),$ since the forms where the two sums differ have $ |\Disc(f)| \asymp Y$ and $t_f \asymp T_1$.

Let
\begin{equation}
 W(f) = \int_{g \in \SL_2(\bR)} \int_0^\infty \sum_{h \in H_{\pm}: f = d_\lambda g h \cdot f_{\pm}}\omega(h \cdot f_{\pm}) F(g)dg \frac{d\lambda}{\lambda}.
\end{equation}

\begin{lemma}\label{weight_function_lemma}
 Let $v \in \bR^4$ be a unit vector.  For $f = d_\lambda n_u a_t k_\theta \cdot f_{\pm}$ with $\lambda \geq 1$ and $t \gg 1$, we have the bound for partial derivatives,
 \begin{equation}
  \partial_v \frac{V_{\frac{1}{2} + i\tau}\left(\frac{|\Disc(f)|}{\sqrt{432}} \right)\sigma\left(\frac{|\Disc(f)|}{Y} \right)W(f)}{|\Disc(f)|^{\frac{1}{2}}} \ll \frac{t^5}{\lambda^3} \left(1 + \frac{|\Disc(f)|}{\tau^2}\right)^{-A}. 
 \end{equation}

\end{lemma}

\begin{proof}
Since we can write $W(f)$ as a group convolution, we can pass the derivative inside the integral, and thus we bound 
 \begin{align}
  \partial_v W(f) &=  \int_{h \in H_{\pm}, d_\lambda gh \cdot f_{\pm} = f} \omega(h \cdot f_{\pm})  \partial_v F(g) dh. 
 \end{align}
 Notice that $H_{\pm}$ is $K$-invariant, so we can eliminate the dependence on $\theta$.  Also, the integral is $\lambda$-invariant, so the derivative depends only on $u$ and $t$.  Thus write
 $\partial_v F = \frac{\partial{F}}{\partial u} \partial_v u + \frac{\partial{F}}{\partial t} \partial_v t$ and bound $\partial_v u \ll \frac{t^5}{\lambda}$, $\partial_v t \ll \frac{t^4}{\lambda}$ to bound $\partial_v W \ll \frac{t^5}{\lambda}$.  The remainder of the claim follows from the estimate for the derivatives of $V$ in Lemma \ref{V_lemma}, together with the bound for the partial derivatives of the discriminant.
\end{proof}
Let 
\begin{align}
 \Sigma_1''(Y) &= \frac{1}{n_{\pm}M_{\pm}} {\sum_{f}}' \sigma\left(\frac{|\Disc(f)|}{Y} \right)\frac{V_{\frac{1}{2}+i\tau}\left(\frac{|\Disc(f)|}{\sqrt{432}} \right)W(f)}{|\Disc(f)|^{\frac{1}{2}}}\E_{y \in B_{R(Y)}}\left[|\Disc(f+y)|^{-i\tau} \right].
\end{align}
Applying the derivative bound,
\begin{equation}\label{derivative_bound}
 |\Sigma_1'(Y)-\Sigma_1''(Y)| \ll Y^{\frac{1}{4}} T_1^5 R \left(1 + \frac{Y}{\tau^2}\right)^{-A}.
\end{equation}
By the bound for exponential sums in Lemma \ref{exponential_sum_lemma}, with $R = \frac{Y^{\frac{1}{4}}}{\tau^{\frac{7}{27}} T_1^{\frac{7}{3}}} $, \begin{equation}|\Sigma_1''(Y)| \ll \frac{T_1^{\frac{8}{3}} \sqrt{Y}}{\tau^{\frac{1}{27}-\epsilon}}\left(1 + \frac{Y}{\tau^2}\right)^{-A}.\end{equation}
Notice that, for the choice of $R$, this dominates the error in (\ref{derivative_bound}).

It follows that
\begin{align}
\Sigma_1(Y)= O\left(\left(\frac{\sqrt{Y}}{T_1} + Y^{\frac{1}{3}} +  \frac{T_1^{\frac{8}{3}}\sqrt{Y}}{\tau^{\frac{1}{27}-\epsilon}}\right)\left(1 + \frac{Y}{\tau^2}\right)^{-A}\right).
\end{align}
Choose, optimally, $T_1 = \min\left(Y^{\frac{1}{24}}, \tau^{\frac{1}{99}} \right)$.  Summed over $Y = 2^n$ obtains a bound of $\Sigma_1 = O\left(\tau^{\frac{98}{99}+\epsilon}\right)$.

We next bound $\Sigma_2$.  
Let $R = o(Y^{\frac{1}{4}}T_2^3)$.  Write $f_d = f+(0,0,0,d)$. Using the Jacobian estimates for change of coordinates,
\begin{equation}
 \frac{\partial t}{\partial d} = O\left(\frac{1}{\lambda t^2} \right), \qquad \frac{\partial \lambda}{\partial d} = O\left(\frac{1}{t^3} \right),
\end{equation}
and hence  for $|d| \leq R$, 
\begin{equation}|\Disc(f_d) - \Disc(f)| = O(RY^{\frac{3}{4}}T_2^{-3}) = o(Y)\end{equation} so $|\Disc(f_d)| \asymp |\Disc(f)|$ and similarly $t_{f_d} \asymp t_f$.  
Let
\begin{equation}
 \Sigma_2'(Y) = \frac{1}{n_{\pm}M_{\pm}} \sum_{\substack{f \in V_\zed \cap V_{\pm}\\ a = 0, t \geq T_2}} \E_{|d| \leq R} \left[\sigma\left(\frac{|\Disc(f_d)|}{Y} \right)\frac{V_{\frac{1}{2} + i\tau}\left(\frac{|\Disc(f_d)|}{\sqrt{432}} \right)W(f_d)}{|\Disc(f_d)|^{\frac{1}{2}+i\tau}} \right] .
\end{equation}
Then $|\Sigma_2'(Y) -\Sigma_2(Y)| = O\left(Y^{\frac{1}{4}}T_2\left(1+\frac{Y}{\tau^2} \right)^{-A} \right)$, since the forms where the two sums differ have $t_f \asymp T_2$ and $|\Disc(f)| \asymp Y$.

We next estimate the derivative of the weight function with respect to $d$.

\begin{lemma}\label{weight_function_lemma_d}
   For $f = d_\lambda n_u a_t k_\theta \cdot f_{\pm}$ with $\lambda \geq 1$ and $t \gg 1$, and $a = 0$.  We have the bound for partial derivatives,
 \begin{equation}
  \partial_d \frac{V_{\frac{1}{2} + i\tau}\left(\frac{|\Disc(f)|}{\sqrt{432}} \right)\sigma\left(\frac{|\Disc(f)|}{Y} \right)W(f)}{|\Disc(f)|^{\frac{1}{2}}} \ll \frac{1}{\lambda^3T_2^3} \left(1 + \frac{|\Disc(f)|}{\tau^2}\right)^{-A}. 
 \end{equation}

\end{lemma}
\begin{proof}
 We have $\frac{\partial}{\partial d} |\Disc(f)| = b^3 = O\left(\left(\frac{\lambda}{T_2}\right)^3 \right)$.  Meanwhile,
 \begin{equation}
  \frac{\partial}{\partial d} W(f) = \frac{\partial}{\partial d} \int_{g \cdot v = f}F(g) \omega(v) \frac{dv}{|\Disc(v)|}.
 \end{equation}
If $d_\lambda n_u a_t \cdot v = f$ then $\frac{\partial}{\partial d} v = O\left(\frac{1}{\lambda t^3} \right)$.  Pass this estimate for the derivative under the integral to obtain the same estimate for the derivative of the integral.  Combining these estimates proves the claim.
\end{proof}

Let 
\begin{equation} 
\Sigma_2''(Y) = \frac{1}{n_{\pm}M_{\pm}} \sum_{\substack{f \in V_\zed \cap V_{\pm}\\ a = 0, t \geq T_2}} \sigma\left(\frac{|\Disc(f)|}{Y} \right)\frac{V_{\frac{1}{2} + i\tau}\left(\frac{|\Disc(f)|}{\sqrt{432}} \right)W(f)}{|\Disc(f)|^{\frac{1}{2}}} \E_{|d| \leq R}\left[|\Disc(f_d)|^{-i\tau} \right].
\end{equation}
By the estimate for the derivative of the weight function,
\begin{equation}
 \left|\Sigma_2'(Y) - \Sigma_2''(Y)\right| \ll \frac{Y^{\frac{1}{4}}R}{T_2^3} \left(1 + \frac{Y}{\tau^2} \right)^{-A}.
\end{equation}
By the estimate for exponential sums in Lemma \ref{exponential_sum_reducibles_lemma}, subject to the constraint $\tau = o(T_2^3 Y^{\frac{1}{4}})$
\begin{equation}
 |\Sigma_2''(Y)| \ll \frac{Y}{\tau^{\frac{1}{3}}T_2^2} \left(1 + \frac{Y}{\tau^2}\right)^{-A}.
\end{equation}

Collecting together the error terms, it follows that $\Sigma_2(Y)$ has the bound
\begin{equation}
 \Sigma_2(Y) \ll \left(Y^{\frac{1}{4}}T_2+\frac{Y^{\frac{3}{4}}}{\tau^{\frac{2}{3}}T_2} + \frac{Y}{\tau^{\frac{1}{3}}T_2^2} \right)\left(1 + \frac{Y}{\tau^2}\right)^{-A}.
\end{equation}
Choose $T_2 = \frac{Y^{\frac{1}{4}}}{\tau^{\frac{1}{9}}}$.  In the constraint $\tau = o(T_2^3 Y^{\frac{1}{4}})$ this entails $\tau^{\frac{4}{3}} = o(Y)$, which is satisfied.  We thus obtain
\begin{equation}
 \Sigma_2(Y) \ll \frac{Y^{\frac{1}{2}}}{\tau^{\frac{1}{9}}} \left(1 + \frac{Y}{\tau^2}\right)^{-A}.
\end{equation}
Summing in $Y=2^n$ obtains $\Sigma_2 \ll \tau^{\frac{8}{9}}$.  Combined with the estimate for $\Sigma_1$, this proves the Theorem.

\bibliographystyle{plain}

\end{document}